\documentclass[10pt]{amsart}
\usepackage{amssymb,amsmath,amsthm,amsfonts,amsopn,url}
\usepackage{amscd,amssymb,amsopn,amsmath,amsthm,graphics,amsfonts,enumerate,verbatim,calc}
\usepackage[all]{xy}
\theoremstyle{plain}
\newtheorem{thm}{Theorem}[section]
\newtheorem*{mt*}{Main Theorem}
\newtheorem*{cj*}{Conjecture}
\newtheorem*{nt*}{Notations}

\newtheorem{prop}[thm]{Proposition}

\newtheorem{cor}{Corollary}
\newtheorem{rem}{Remark}
\newtheorem{example}{Example}[section]
\theoremstyle{definition}

\newcommand{\ideal}[1]{\mathfrak{#1}}
\newcommand{\M}{\ideal{M}}
\newcommand{\m}{\ideal{m}}
\newcommand{\n}{\ideal{n}}

\newcommand{\func}[1]{\mathrm{#1} \,}

\newcommand{\Ass}{\func{Ass}}

\newcommand{\rad}{\func{rad}}

\newcommand{\Ker}{\func{Ker}}
\newcommand{\Coker}{\func{Coker}}

\newcommand{\ZZ}{{\mathbb Z}}

\title[]{A Note on Associated Primes and Bockstein Homomorphisms of Local Cohomology Modules for Ramified Regular Local Rings}

\author[]{Rajsekhar Bhattacharyya}

\address{Dinabandhu Andrews College, Garia, Kolkata 700084, India}

\email{rbhattacharyya@gmail.com}

\thanks{}

\keywords{Local Cohomology}

\subjclass[2010]{13D45}
\begin{document}

\begin{abstract}
For a Noetherian regular ring $S$ and for a fixed ideal $J\subset S$, assume that the associated primes of local cohomology module $H^i_J(S)$ does not contain $p$ for some $i\geq 0$, and we call this as a property $\textit{\textbf{P}}^{i,p}_{J,S}$ or $\textit{\textbf{P}}$ for brevity. Recently, in Theorem 1.2 of \cite{Nu1}, it is proved that in a Noetherian regular local ring $S$ and for a fixed ideal $J\subset S$, associated primes of local cohomology module $H^i_J(S)$ for $i\geq 0$ is finite, if it does not contain $p$. In this paper, we study how the property $\textit{\textbf{P}}$ (as mentioned above) can come down from unramified regular ring to ramified regular local ring. In \cite{SW}, Bockstein homomorphism is studied in the context to the finiteness of associated primes of local cohomology modules for the ring of integers. There it is shown that if $p$ is nonzero divisor of Koszul homology then Bockstein homomorphism is a zero map (see, Theorem 3.1 of \cite{SW}). Here, in this paper, as a consequence of property $\textit{\textbf{P}}$, we extend the result of Theorem 3.1 of \cite{SW} to the ramified regular local ring.
\end{abstract}

\maketitle

\section{introduction}

In Commutative Algebra, an important direction of research is to determine when the set of associated primes of the $i$-th local cohomology module $H^i_I (R)$ with support in ideal $I\subset R$ is finite. There are several important cases where we have the finiteness of associated primes of local cohomology modules and we list them below: (1) For regular rings of prime characteristic \cite{HS}, (2) for regular local and affine rings of characteristic zero \cite{Ly1}, (3) for unramified regular local rings of mixed characteristic \cite{Ly2} and (4) for smooth algebra over $\ZZ$ and $\ZZ_{p\ZZ}$ \cite{BBLSZ}. These results support Lyubeznik's conjecture, (\cite{Ly1}, Remark 3.7): 

\begin{cj*} Let $R$ be a regular ring and $I\subset R$ be its ideal, then for each $i\geq 0$, $i$-th local cohomology module $H^i_I (R)$ has finitely many associated prime ideals.
\end{cj*}

The conjecture is open for the ramified regular local ring, in fact, more specifically, it is open only for the primes that contain $p$ (see \cite{Nu1}). Recently, in \cite{Pu2}, it is shown that for excellent regular ring $R$ of dimension $d$, containing a field of characteristic zero, for an ideal $I\subset R$, $\Ass_R H^{d-1}_I (R)$ is finite. 

Finiteness condition of associated primes is very deeply related to the $p$-torsion elements. In another way, it is important to study when a prime integer $p$ in the ring does not belong to the set of associated primes of local cohomology module. More precisely, for a Noetherian regular ring $S$ and for a fixed ideal $J\subset S$, if associated primes of local cohomology $H^i_J(S)$ does not contain $p$ for some $i\geq 0$, then we call this as a property $\textit{\textbf{P}}^{i,p}_{J,S}$ i.e. `$\textit{\textbf{P}}^{i,p}_{J,S}\equiv p\notin \Ass_{\ZZ} H^i_J(S)$' and we will write only $\textit{\textbf{P}}$ for brevity or when integers $i$, $p$, ring $S$ and ideal $J$ is clear from the context. Recently, in Theorem 1.2 of \cite{Nu1}, it is proved that in a Noetherian regular local ring $S$ and for a fixed ideal $J\subset S$, associated primes of local cohomology $H^i_J(S)$ for $i\geq 0$ is finite, if it does not contain $p$. Since associated primes of local cohomology modules of unramified regular local ring is finite we are only interested to investigate the property $\textit{\textbf{P}}$ only for the ramified regular local ring. 

It is well known fact that every ramified regular local ring is a typical extension of unramified regular local ring, which is known as Eisenstein extension. Here, in this paper, we consider Eisenstein ring extension as well as the ring extension $V[[X,Y]][[Z_1,\ldots,Z_n]]\rightarrow S$ (see page 2 of \cite{Nu2}), where $S=V[[X,Y]][[Z_1,\ldots,Z_n]]/(XY-p)$ and where $V$ is complete DVR of mixed characteristic $p>0$. In section 2 of this paper, we study how the property $\textit{\textbf{P}}$ (as mentioned above) can come down to ramified regular local ring under the ring extensions mentioned above. 

Consider an arbitrary Noetherian ring $R$ and let $I$ be an ideal of it. For a nonzero divisor $u$, consider the short exact sequence $$0\rightarrow R\stackrel{u}{\rightarrow} R\rightarrow R/uR \rightarrow 0,$$ which leads to the following long exact sequence of local cohomologies $$..H^{i-1}_{I}(R/uR)\stackrel{\delta^{i-1}_{u,I}}\rightarrow H^{i}_{I}(R)\stackrel{u}{\rightarrow}H^{i}_{I}(R)\stackrel{\pi^i_{u,I}}\rightarrow H^{i}_{I}(R/uR) \rightarrow ..$$

We define Bockstein homomorphism as the map $$\beta^{i-1}_{u,I}=\pi^i_{u,I}\circ\delta^{i-1}_{u,I}: H^{i-1}_I(R/uR)\rightarrow H^{i}_I(R/uR).$$ In \cite{SW}, Bockstein homomorphism is studied in the context to study finiteness of associated primes of local cohomology modules for the ring of integers. There, it is shown that if $p$ is nonzero divisor of Koszul homology then Bockstein homomorphism is a zero map (see, Theorem 3.1 of \cite{SW}). Here, in section 3, we extend this result to the ramified regular local ring. In section 4, we give examples of when the property `$\textit{\textbf{P}}^{i,p}_{J,S}\equiv p\notin \Ass_{\ZZ} H^i_J(S)$' is satisfied. In Example 4.1 of \cite{SW}, it is shown that for a certain ramified regular ring, Bockstein homomorphism is not always zero. In section 5 of this paper, we study a few examples of Bockstein homomorphism for ramified regular local rings and in this context, we apply the results of section 3 to look around the situations of Example 4.1 of \cite{SW} locally. We also observe when Bockstein homomorphism of Example 4.1 of \cite{SW}, is zero.

In this paper, all Noetherian rings are $\ZZ$ algebras and are denoted by $R$, $S$. In each case, prime integer $p$ also generates prime in $R$ and $S$ (as for example if $R$ is unramified regular local ring in mixed characteristic $p>0$ the $p$ is a regular parameter and $R/pR$ is again a regular ring and hence $pR$ is a prime). Thus for any $R$-module $M$, if $p\notin \Ass_{\ZZ} M$ then $pR\notin \Ass_{R} M$. Similar is true for $S$-module $M$. 


\section{behaviour of finiteness of associated primes of local cohomology modules for ramified regular local rings}

For local ring $(R,\m,K)$ of mixed characteristic, let characteristics of $K=p$. We say $R$ is unramified if $p\notin {\m}^2$ and it is ramified if $p\in {\m}^2$. For normal local ring $(R,\m)$, consider the extension ring defined by $S=R[X]/f(X)=R[x]$ where $f(X)=X^n+ a_1X^{n-1}+\ldots +a_n$ with $a_i\in \m$ for every $i\geq 0$ and $a_n\notin {\m}^2$. This ring $(S,\n,K)$ is local and it is defined as an Eisenstein extension of $R$ and $f(X)$ is known as an Eisenstein polynomial (see page 228-229 of \cite{CRTofM}). More over from paragraph above Theorem 29.8 of \cite{CRTofM} we get that $\n= \m S+ xS \supset \m R[x] \supset {\m}^2$ and for $p\in \m$, $p\in {\n}^2$. Thus if $R$ is an unramified local ring then $S$ is a ramified local ring. 

Here we note down two important results regarding Eisenstein extensions:

(1) An Eisenstein extension of regular local ring is regular local (see Theorem 29.8 (i) of \cite{CRTofM}). Thus from above discussion we can observe that an Eisenstein extensions of an unramified regular local ring is a ramified regular local ring.

(2) Every ramified regular local ring is an Eisenstein extension of some unramified regular local ring (see Theorem 29.8 (ii) of \cite{CRTofM})

\begin{nt*}

Here we adopt the following notations: For Noetherian ring $R$ and its ideal $I=(f_1,\ldots,f_t)\subset R$, we write the symbols $H^{i}(\textbf{f}:R)$ and $H^{i}_{I}(R)$ for $i$-th Koszul cohomology and $i$-th local cohomology.
 
\end{nt*}

At first, we observe the following elementary result.

\begin{prop}
Let $R$ be a domain of characteristic $0$ and $M$ be an $R$-module. For a prime integer $p\in \ZZ$, assume that $p\notin \Ass_{\ZZ} M$. Then for every monic polynomial $f(X)\in R[X]$, $(p,f(X))$ and $(f(X),p)$ both form $M[X]$-regular sequences. In particular, both of them are $R[X]$-regular sequences as well.
\end{prop}

\begin{proof}
Choose a monic polynomial $f(X)=X^n+ a_1X^{n-1}+\ldots +a_n$ in $R[X]$ and consider $M[X]/pM[X]=(M/pM)[X]$ as an $(R/p)[X]$-module. Let $0\neq\overline{z}=\overline{z_0} + \overline{z_1}X +\ldots +\overline{z_t}X^t\in (M/pM)[X]$ with $\overline{z_t}\neq 0$. Now $f(X)\overline{z}=\overline{f(X)}\overline{z}=0$ implies $\overline{z_t}X^{t+n}=0$. Thus $\overline{z_t}=0$ and we have a contradiction. Thus $(p, f(X))$ is a $M[X]$-regular sequence. Since $f(X)$ is itself $M[X]$-regular, by (b) of (12.2) Discussion of \cite{Ho}, we find $(f(X), p)$ is also a $M[X]$-regular sequence. Second assertion is immediate
.
\end{proof}

From above discussion we find that a ramified regular local ring is a homomorphic image of a polynomial ring in one variable over some unramified regular local ring. In the following theorem we observe, how the property $\textit{\textbf{P}}$ comes down from polynomial ring to the ramified regular local ring.

\begin{thm}
Let $S$ be a ramified regular local ring in mixed characteristic $p>0$ and let $J\subset S$ be an ideal. Suppose, there exists an unramified regular local ring $R$ such that $S$ is an Eisenstein extension of $R$ by an Eisenstein polynomial $f(X)\in R[X]$. Let $I$ be an ideal of $R[X]$ such that $IS= J$. Assume that $p\notin \Ass_{\ZZ} H^{i}_{I}(R[X])\cup \Ass_{\ZZ} H^{i+1}_{I}(R[X])$. Then $p\notin \Ass_{\ZZ} H^i_J(S)$ and $\Ass_S H^i_J(S)$ is finite. 
\end{thm}

\begin{proof}
Here $S=R[X]/f(X)$ where $(R,\m)$ is an unramified regular local ring and $f(X)$ be an Eisenstein polynomial. Since $f(X)$ is monic, from above proposition we get that $(p,f(X))$ and $(f(X),p)$ are both $R[X]$-regular sequences as well as $H^{i}_{I}(R[X])$-regular sequences. 

Consider the following diagram of short exact sequences where all the rows and columns are exact.

\[
\CD
@. 0@. 0@. 0@. @. \\
@. @VVV @VVV @VVV @.\\
0 @>>>R'@>f>>R'@>>>R'/fR'@>>>0 @.\\
@. @Vp VV @VVp V @VVp V @. \\
0 @>>>R'@>f>>R'@>>>R'/fR'@>>>0 @.\\
@. @VVV @VVV @VVV @. \\
0@>>>R'/pR'@>f>>R/pR'@>>>R'/(p,f)R'@>>>0 @.\\
@. @VVV @VVV @VVV @.\\
@. 0@. 0@. 0@. @. 
\endCD
\]

This above diagram yields the following diagram of long exact sequences where all the rows and columns are exact.

\[
\CD
@. @. @. @. @. H^i_{I}(R')@.\\
@.@VVV @VVV @VVV @VVV @Vp VV \\
@. H^{i-1}_{I}(R') @>f>>H^{i-1}_{I}(R')@>\psi_{i-1}>>H^{i-1}_{I}(R'/fR')@>>>H^i_{I}(R')@>f>>H^i_{I}(R')@.\\
@. @V\pi_{i-1} VV @V\pi_{i-1} VV @V\alpha_{i-1} VV @V\pi_{i} VV @V\pi_{i} VV\\
@. H^{i-1}_{I}(R'/pR') @>f>>H^{i-1}_{I}(R'/pR')@>\phi_{i-1}>>H^{i-1}_{I}((R'/(p,f)R')@>>>H^i_{I}(R'/pR')@>f>>H^i_{I}(R'/pR')@.\\
@. @VVV @VVV @VVV @VVV @. \\
@. H^{i}_{I}(R') @>f>>H^{i}_{I}(R')@>>>H^{i}_{I}(R'/fR')@>>>H^{i+1}_{I}(R') @. @.\\
@. @Vp VV @Vp VV @Vp VV @Vp VV @. \\
@. H^{i}_{I}(R') @>f>>H^{i}_{I}(R')@>>>H^{i}_{I}(R'/fR')@>>>H^{i+1}_{I}(R') @. @.\\
@. @VVV @VVV @VVV @VVV @. 
\endCD
\]

Since $\pi_{i-1}$ is surjective, from exactness $f\circ\pi_{i-1}$ maps into $\Ker(\phi_{i-1})$. Along with the injectivity of $H^i_{I}(R')\stackrel{f}{\rightarrow}H^i_{I}(R')$, we also claim the injectivity of $H^i_{I}(R'/pR') \stackrel{f}{\rightarrow}H^i_{I}(R'/pR')$. To see this consider an element $x\in H^i_{I}(R'/pR')$ such that $fx=0$. Since $\pi_i$ is surjective there exists some $z\in H^{i}_{I}(R')$ such that $\pi_i(z)=x$. From commutativity of the square in the top-right corner, we get $fz= py$ since $\pi_i(fz)=0$. Since $(p,f)$ is a weakly $H^{i}_{I}(R')$-regular sequence, there exists $z'\in H^{i}_{I}(R')$ such that $z= pz'$. This implies, $x=\pi_i(z)= \pi_i(pz')=0$ from exactness. Thus, we get the diagram below where rows are exact.

\[
\CD
@.H^{i-1}_{I}(R') @>f>>H^{i-1}_{I}(R')@>\psi_{i-1}>>H^{i-1}_{I}(R'/fR')@>>>0\\
@. @V{f\circ\pi_{i-1}} VV @VV\pi_{i-1} V @V\alpha_{i-1} VV \\
0@>>>\Ker(\phi_{i-1}) @>>>H^{i-1}_{I}(R'/pR')@>\phi_{i-1}>>H^{i-1}_{I}((R'/(p,f)R')@>>>0 @.
\endCD
\]

Since $\Coker(\pi_{i-1})=0$, using snake lemma we have $\Coker(\alpha_{i-1})=0$. This implies that $\alpha_{i-1}$ is surjective and $p$ is a nonzero divisor of $H^{i}_{IR[X]}(R[X]/fR[X])=H^{i}_{I(R[X]/fR[X])}(R[X]/fR[X])=H^{i}_{J}(S)$. Now using Theorem 1.2 of \cite{Nu1} we conclude. 
\end{proof}

Cohen structure theorem tells us that any ramified regular local ring $S$ can be written as $S= R/(h-p)R$ for some unramified regular local ring $R$, where $h$ is in the square of the maximal ideal of $R$. In several papers of Luis Nunez-Betancourt (as for example see \cite{Nu2}) the example of the following ramified regular local ring $S=V[[X,Y]][[Z_1,\ldots,Z_n]]/(XY-p)$ is given, where $V$ is complete DVR of mixed characteristic $p>0$. It is being claimed as the one of the simplest example of ramified regular local ring for which the finiteness of associated primes of its local cohomology is unknown. 

In the above context, in the next theorem, we observe partial results about the finiteness of associated primes of local cohomology modules of the ramified regular local ring $S=V[[X,Y]][[Z_1,\ldots,Z_n]]/(XY-p)$ as mentioned above, which can be obtained from an unramified regular local ring $R=V[[Z_1,\ldots,Z_n]]$.

\begin{thm}
Let $R$ be an unramified regular local ring in mixed characteristic $p>0$ and $T=R[[X,Y]]/(XY-p)$ be a ramified local ring. Let $I$ be an ideal of $R$. For an integer $i\geq 0$, assume that $p\notin \Ass_{\ZZ} H^{i}_I(R) \cup \Ass_{\ZZ} H^{i+1}_{I}(R)$. Then for the ramified regular local ring $T$, $p\notin \Ass_{\ZZ} H^i_{IT}(T)$ and $\Ass_T H^{i}_{IT}(T)$ is finite.
\end{thm}

\begin{proof}
Let $(R,\m)$ be an unramified regular local ring in mixed characteristic $p>0$. For an ideal $I\subset R$ and for an integer $i\geq 0$, we assume $p\notin \Ass_{\ZZ} H^{i}_I(R) \cup \Ass_{\ZZ} H^{i+1}_{I}(R)$. Consider the ring $R[X,Y]$ and $R[X,Y]$-module $H^{i}_{IR[X,Y]}(R[X,Y])$.From hypothesis, $p\notin \Ass_{\ZZ} H^{i}_{IR[X,Y]}(R[X,Y]) \cup \Ass_{\ZZ} H^{i+1}_{IR[X,Y]}(R[X,Y])$. Set $M= H^{i}_I(R)$, $M[X,Y]=H^{i}_{IR[X,Y]}(R[X,Y])$ and polynomial $f(X,Y)=XY-p$. 

Consider $M[X,Y]/pM[X,Y]=(M/pM)[X,Y]$ as an $(R/p)[X,Y]$-module. Let $0\neq\overline{z}=\sum^{(a,b)}_{(i,j)=(0,0)}\overline{z_{i,j}}X^iY^j $ with $\overline{z_{a,b}}\neq 0$. Now $f(X,Y)\overline{z}=\overline{f(X,Y)}\overline{z}=0$ implies $\overline{z_{a,b}}X^{a+1}Y^{b+1}=0$. Thus $\overline{z_{a,b}}=0$ and we have a contradiction. Thus $(p, f(X,Y))$ is a $M[X,Y]$-regular sequence. Since $f(X,Y)$ is itself $M[X,Y]$-regular, by (b) of (12.2) Discussion of \cite{Ho}, we find $(f(X,Y), p)$ is also a $M[X,Y]$-regular sequence. Similar argument shows that $(f(X,Y), p)$ is also a $R[X,Y]$-regular sequence. Set $f(X,Y)=f$, $R[X,Y]=R'$ and $IR[X,Y]=I'$.

Consider the following diagram of short exact sequences where all the rows and columns are exact.

\[
\CD
@. 0@. 0@. 0@. @. \\
@. @VVV @VVV @VVV @.\\
0 @>>>R'@>f>>R'@>>>R'/fR'@>>>0 @.\\
@. @Vp VV @VVp V @VVp V @. \\
0 @>>>R'@>f>>R'@>>>R'/fR'@>>>0 @.\\
@. @VVV @VVV @VVV @. \\
0@>>>R'/pR'@>f>>R/pR'@>>>R'/(p,f)R'@>>>0 @.\\
@. @VVV @VVV @VVV @.\\
@. 0@. 0@. 0@. @. \endCD
\]

This above diagram yields the following diagram of long exact sequences where all the rows and columns are exact.

\[
\CD
@. @. @. @. @. H^i_{I'}(R')@.\\
@.@VVV @VVV @VVV @VVV @Vp VV \\
@. H^{i-1}_{I'}(R') @>f>>H^{i-1}_{I'}(R')@>\psi_{i-1}>>H^{i-1}_{I'}(R'/fR')@>>>H^i_{I'}(R')@>f>>H^i_{I'}(R')@.\\
@. @V\pi_{i-1} VV @V\pi_{i-1} VV @V\alpha_{i-1} VV @V\pi_{i} VV @V\pi_{i} VV\\
@. H^{i-1}_{I'}(R'/pR') @>f>>H^{i-1}_{I'}(R'/pR')@>\phi_{i-1}>>H^{i-1}_{I'}((R'/(p,f)R')@>>>H^i_{I'}(R'/pR')@>f>>H^i_{I'}(R'/pR')@.\\
@. @VVV @VVV @VVV @VVV @. \\
@. H^{i}_{I'}(R') @>f>>H^{i}_{I'}(R')@>>>H^{i}_{I'}(R'/fR')@>>>H^{i+1}_{I'}(R') @. @.\\
@. @Vp VV @Vp VV @Vp VV @Vp VV @. \\
@. H^{i}_{I'}(R') @>f>>H^{i}_{I'}(R')@>>>H^{i}_{I'}(R'/fR')@>>>H^{i+1}_{I'}(R') @. @.\\
@. @VVV @VVV @VVV @VVV @. 
\endCD
\]
Since $\pi_{i-1}$ is surjective, from exactness $f\circ\pi_{i-1}$ maps into $\Ker(\phi_{i-1})$. Along with the injectivity of $H^i_{I'}(R')\stackrel{f}{\rightarrow}H^i_{I'}(R')$, we also claim the injectivity of $H^i_{I'}(R'/pR') \stackrel{f}{\rightarrow}H^i_{I'}(R'/pR')$. To se this consider an element $x\in H^i_{I'}(R'/pR')$ such that $fx=0$. Since $\pi_i$ is surjective there exists some $z\in H^{i}_{I'}(R')$ such that $\pi_i(z)=x$. From commutativity of the square in the top-right corner, we get $fz= py$ since $\pi_i(fz)=0$. Since $(p,f)$ is a $H^{i}_{I'}(R')$-regular sequence, there exists $z'\in H^{i}_{I'}(R')$ such that $z= pz'$. This implies, $x=\pi_i(z)= \pi_i(pz')=0$ from exactness. Thus, we get the diagram below where rows are exact.

\[
\CD
@.H^{i-1}_{I'}(R') @>f>>H^{i-1}_{I'}(R')@>\psi_{i-1}>>H^{i-1}_{I'}(R'/fR')@>>>0\\
@. @V{f\circ\pi_{i-1}} VV @VV\pi_{i-1} V @V\alpha_{i-1} VV \\
0@>>>\Ker(\phi_{i-1}) @>>>H^{i-1}_{I'}(R'/pR')@>\phi_{i-1}>>H^{i-1}_{I'}((R'/(p,f)R')@>>>0 @.
\endCD
\]

Since $\Coker(\pi_{i-1})=0$, using snake lemma we have $\Coker(\alpha_{i-1})=0$. This implies that $\alpha_{i-1}$ is surjective and $p$ is a nonzero divisor of $H^{i}_{IR[X,Y]}(R[X,Y]/fR[X,Y])=H^{i}_{I(R[X,Y]/fR[X,Y])}(R[X,Y]/fR[X,Y])=H^{i}_{IS}(S)$. Consider the completion of $S$ by its maximal ideal $(X,Y)$ and we get $\hat{S}=R[[X,Y]]/fR[[X,Y]]$ where $S\rightarrow \hat{S}$ is faithfully flat. Moreover due to flat extension, $p$ is also $H^{i}_{I\hat{S}}(\hat{S})$-regular. Set $\hat{S}=T$ and using Theorem 1.2 of \cite{Nu1} we get $\Ass_T H^{i}_{IT}(T)$ is finite.
\end{proof}

\section{behaviour of bockstein homomorphisms of local cohomology modules for ramified regular local rings}

In \cite{SW}, Bockstein homomorphism is studied for polynomial ring over $\ZZ$. There, it is shown that if prime integer $p$ is a nonzero divisor of Koszul cohomology, then Bockstein homomorphism between the local cohomology modules is a zero map. For detail see Theorem 3.1 of \cite{SW}. In this section, we study the behaviour of Bockstein homomorphism for unramified and ramified regular local ring with the help of the results of previous section. We begin the section with the following examples of unramified regular local ring.

\begin{example}

Let $V$ be a DVR in mixed characteristics whose maximal ideal is generated by prime $p$. From such a DVR, we can always construct an unramified regular local ring as described in Theorem IV.7 of \cite{Ro}: With $V$, adjoin indeterminates $X_1,\ldots,X_n$ and then localize it by maximal ideal $\M$ generated by $p$ and all the indeterminates. Clearly $V[X_1,\ldots,X_n]$ as well as $(V[X_1,\ldots,X_n])_{\M}$ are smooth $V$-algebra as described in \cite{BBLSZ} which is also known as regular algebra in \cite{CRTofM}. 

\end{example}

We observe the following remark before the main results of this section.

\begin{rem}
In the statements of the following corollaries, we have use the equivalence between the conditions `$p\notin \Ass_{\ZZ} H^{i}(\textbf{f};R) \cup \Ass_{\ZZ} H^{i+1}(\textbf{f};R)$' and `$p\notin \Ass_{\ZZ} H^{i}(\textbf{f},0;R)$'. It is a well known fact that (as for example, see Remark 1.4 of \cite{Hu}), $H^{i}(\textbf{f};R) \oplus H^{i+1}(\textbf{f};R)= H^{i}(\textbf{f},0;R)$ and from this, equivalence follows at once. 
\end{rem}

From Theorem 2.2, we have an immediate corollary, which extends the result of Bockstein homomorphism of Theorem 3.1 of \cite{SW} for ramified regular local ring.

\begin{cor}
Let $S$ be a ramified regular local ring in mixed characteristic $p>0$ and let $J\subset S$ be an ideal. Suppose, there exists an unramified regular local ring $R$ of Example 3.1, such that $S$ is an Eisenstein extension of $R$ by an Eisenstein polynomial $f(X)\in R[X]$. Let $I=(f_1,\ldots,f_t) R[X]$ be an ideal of $R[X]$ such that $IS= J$. Assume that $p\notin \Ass_{\ZZ} H^{i}(\textbf{f},0;R[X])$ (equivalently, $p\notin \Ass_{\ZZ} H^{i}(\textbf{f};R[X])\cup \Ass_{\ZZ} H^{i+1}(\textbf{f};R[X])$). Then Bockstein homomorphism $\beta^{i-1}_{p,J}: H^{i-1}_J(S/pS)\rightarrow H^{i}_J(S/pS)$ is zero.
\end{cor}

\begin{proof}
According to the example, $R=(V[X_1,\ldots,X_n])_{\M}$. Clearly $V[X_1,\ldots,X_n]$ as well as $(V[X_1,\ldots,X_n])_{\M}$ are smooth $V$-algebra as described in \cite{BBLSZ}. Moreover adjoining indeterminate to $R$, $R[X]$ also becomes smooth over $V$. So, according to Theorem 4.1 of \cite{BBLSZ}, $p\notin \Ass_{\ZZ} H^{i}_{I}(R[X])\cup \Ass_{\ZZ} H^{i+1}_{I}(R[X])$. Thus using Theorem 2.2 we get that $p$ is a nonzero divisor of $H^{i}_J(S)$ and from the second diagram in the proof of Theorem 2.2 we can conclude.
\end{proof}

From Theorem 2.3, we have the following corollary, which extends the result of Bockstein homomorphism of Theorem 3.1 of \cite{SW}, even for ramified regular local ring $S=V[[X,Y]][[Z_1,\ldots,Z_n]]/(XY-p)$ as given in \cite{Nu2}.

\begin{cor}
Let $R$ be an unramified regular local ring in mixed characteristic $p>0$ and $T=R[[X,Y]]/(XY-p)$ be a ramified local ring. Let $I$ be an ideal of $R$. Set the generators of $I$ such that $I=(f_1,\ldots,f_t)$. For an integer $i\geq 0$, assume that $p\notin \Ass_{\ZZ} H^{i}(\textbf{f},0;R)$' (equivalently, $p\notin \Ass_{\ZZ} H^{i}(\textbf{f};R) \cup \Ass_{\ZZ} H^{i+1}(\textbf{f};R)$). Then for the ramified regular local ring $T$, Bockstein homomorphism $\beta^{i-1}_{p,IT}: H^{i-1}_{IT}(T/pT)\rightarrow H^{i}_{IT}(T/pT)$ is zero.
\end{cor}

\begin{proof}
Using Theorem 4.1 of \cite{BBLSZ} we find that $p\notin \Ass_{\ZZ} H^{i}_I(R) \cup \Ass_{\ZZ} H^{i+1}_{I}(R)$. The rest of the proof is similar to those of above corollaries. 
\end{proof}

For unramified regular local ring $R$ and its ideal $I$, let $S$ be a ramified regular local ring which is an Eisenstein extension of $R$. Let $J$ be an ideal of $S$ such that $\rad (J+p)= \rad (IS+p)$. Then the property $\textit{\textbf{P}}^{i,p}_{I,R}$ as stated in the introduction implies that Bockstein homomorphism for $\beta^{i-1}_{p,J}$ for $H^{i-1}_J$. is a zero map. In doing so we need the help of Lemma 2.4 of \cite{SW}. 

\begin{thm}
Let $R$ be an unramified regular local ring in mixed characteristic $p>0$. Let $S$ be any ramified regular local ring which can obtained from $R$ via an Eisenstein extension. Let $I$ be an ideal of $R$. For every integer $i\geq 0$, assume that $p\notin \Ass_{\ZZ} H^{i}_I(R) \cup \Ass_{\ZZ} H^{i+1}_{I}(R)$. Then, for every ideal $J\subset S$ such that $\rad (J+p)= \rad (IS+p)$, Bockstein homomorphism $\beta^{i-1}_p: H^{i-1}_J(S/pS)\rightarrow H^{i}_J(S/pS)$ is zero. 
\end{thm}

\begin{proof}
Let $(R,\m)$ be an unramified regular local ring in mixed characteristic $p>0$. For an ideal $I\subset R$ and for an integer $i\geq 0$, we assume $p\notin \Ass_{\ZZ} H^{i}_I(R) \cup \Ass_{\ZZ} H^{i+1}_{I}(R)$. Consider the ring $R[X]$ and $R[X]$-module $H^{i}_{IR[X]}(R[X])$. From hypothesis, $p\notin \Ass_{\ZZ} H^{i}_{IR[X]}(R[X]) \cup \Ass_{\ZZ} H^{i+1}_{IR[X]}(R[X])$. Let $f(X)$ be the Eisenstein polynomial. Since $f(X)$ is monic, from Proposition 2.1 we get that $(p,f(X))$ and $(f(X),p)$ are both $R[X]$-regular sequences as well as $H^{i}_{I}(R[X])$-regular sequences. 
Set $f(X)=f$, $R[X]=R'$ and $IR[X]=I'$.

Consider the following diagram of short exact sequences where all the rows and columns are exact.
\[
\CD
@. 0@. 0@. 0@. @. \\
@. @VVV @VVV @VVV @.\\
0 @>>>R'@>f>>R'@>>>R'/fR'@>>>0 @.\\
@. @Vp VV @VVp V @VVp V @. \\
0 @>>>R'@>f>>R'@>>>R'/fR'@>>>0 @.\\
@. @VVV @VVV @VVV @. \\
0@>>>R'/pR'@>f>>R/pR'@>>>R'/(p,f)R'@>>>0 @.\\
@. @VVV @VVV @VVV @.\\
@. 0@. 0@. 0@. @. \endCD
\]

This above diagram yields the following diagram of long exact sequences where all the rows and columns are exact.

\[
\CD
@. @. @. @. @. H^i_{I'}(R')@.\\
@.@VVV @VVV @VVV @VVV @Vp VV \\
@. H^{i-1}_{I'}(R') @>f>>H^{i-1}_{I'}(R')@>\psi_{i-1}>>H^{i-1}_{I'}(R'/fR')@>>>H^i_{I'}(R')@>f>>H^i_{I'}(R')@.\\
@. @V\pi_{i-1} VV @V\pi_{i-1} VV @V\alpha_{i-1} VV @V\pi_{i} VV @V\pi_{i} VV\\
@. H^{i-1}_{I'}(R'/pR') @>f>>H^{i-1}_{I'}(R'/pR')@>\phi_{i-1}>>H^{i-1}_{I'}((R'/(p,f)R')@>>>H^i_{I'}(R'/pR')@>f>>H^i_{I'}(R'/pR')@.\\
@. @VVV @VVV @VVV @VVV @. \\
@. H^{i}_{I'}(R') @>f>>H^{i}_{I'}(R')@>>>H^{i}_{I'}(R'/fR')@>>>H^{i+1}_{I'}(R') @. @.\\
@. @Vp VV @Vp VV @Vp VV @Vp VV @. \\
@. H^{i}_{I'}(R') @>f>>H^{i}_{I'}(R')@>>>H^{i}_{I'}(R'/fR')@>>>H^{i+1}_{I'}(R') @. @.\\
@. @VVV @VVV @VVV @VVV @. 
\endCD
\]
Since $\pi_{i-1}$ is surjective, from exactness $f\circ\pi_{i-1}$ maps into $\Ker(\phi_{i-1})$. Along with the injectivity of $H^i_{I'}(R')\stackrel{f}{\rightarrow}H^i_{I'}(R')$, we also claim the injectivity of $H^i_{I'}(R'/pR') \stackrel{f}{\rightarrow}H^i_{I'}(R'/pR')$. To se this consider an element $x\in H^i_{I'}(R'/pR')$ such that $fx=0$. Since $\pi_i$ is surjective there exists some $z\in H^{i}_{I'}(R')$ such that $\pi_i(z)=x$. From commutativity of the square in the top-right corner, we get $fz= py$ since $\pi_i(fz)=0$. Since $(p,f)$ is a $H^{i}_{I'}(R')$-regular sequence, there exists $z'\in H^{i}_{I'}(R')$ such that $z= pz'$. This implies $x=\pi_i(z)= \pi_i(pz')=0$ from exactness. Thus, we get the diagram below where rows are exact.

\[
\CD
@.H^{i-1}_{I'}(R') @>f>>H^{i-1}_{I'}(R')@>\psi_{i-1}>>H^{i-1}_{I'}(R'/fR')@>>>0\\
@. @V{f\circ\pi_{i-1}} VV @VV\pi_{i-1} V @V\alpha_{i-1} VV \\
0@>>>\Ker(\phi_{i-1}) @>>>H^{i-1}_{I'}(R'/pR')@>\phi_{i-1}>>H^{i-1}_{I'}((R'/(p,f)R')@>>>0 @.
\endCD
\]

Since $\Coker(\pi_{i-1})=0$, using snake lemma we have $\Coker(\alpha_{i-1})=0$. This implies that $\alpha_{i-1}$ is surjective and $p$ is a nonzero divisor of $H^{i}_{IR[X]}(R[X]/fR[X])=H^{i}_{I(R[X]/fR[X])}(R[X]/fR[X])=H^{i}_{IS}(S)$. This implies that Bockstein homomorphism $H^{i-1}_{IS}(S/pS)\rightarrow H^{i}_{IS}(S/pS)$ is zero. 

Now observe Lemma 2.4 of \cite{SW} and consider the following commutative diagram  
\[
  \xymatrix
{
  & H^{i-1}_{IS}(S/pS) 
    \ar@{->}[r]
 & H^{i}_{IS}(S/pS) 
 \\
  & H^{i-1}_{J}(S/pS) 
	\ar@{->}[u]
\ar@{->}[r]
 & H^{i}_{J}(S/pS)
  \ar@{->}[u]      
 }
\]

where each horizontal map is Bockstein homomorphism and vertical maps are isomorphism. Since upper horizontal map is a zero map, we can conclude.
\end{proof}

Although it is not clear whether we can have $\textit{\textbf{P}}^{i,p}_{J,S}$ for ramified regular local ring $S$ follows from $\textit{\textbf{P}}^{i,p}_{I,R}$ where $J\cap R=I$. 

If the unramified regular local ring $R$ is of Example 3.1, then the following corollary follows from above theorem.

\begin{cor}
Let $R$ be an unramified regular local ring in mixed characteristic $p>0$ of above Example 3.1. For an ideal $I=(f_1,\ldots,f_t)\subset R$ and for an integer $i\geq 0$, assume that $p\notin \Ass_{\ZZ} H^{i}(\textbf{f},0;R)$' (equivalently, $p\notin \Ass_{\ZZ} H^{i}(\textbf{f};R) \cup \Ass_{\ZZ} H^{i+1}(\textbf{f};R)$). Let $S$ be any ramified regular local ring which can be obtained from $R$ via an Eisenstein extension. Then, for every ideal $J\subset S$ such that $\rad (J+p)= \rad (IS+p)$, Bockstein homomorphism $\beta^{i-1}_p: H^{i-1}_J(S/pS)\rightarrow H^{i}_J(S/pS)$ is zero.
\end{cor}

\begin{proof}
According to the example, $R=(V[X_1,\ldots,X_n])_{\M}$. Clearly $V[X_1,\ldots,X_n]$ as well as $(V[X_1,\ldots,X_n])_{\M}$ are smooth $V$-algebra as described in \cite{BBLSZ}. So, according to Theorem 4.1 of \cite{BBLSZ}, $p\notin \Ass_{\ZZ} H^{i}_{I}(R[X])\cup \Ass_{\ZZ} H^{i+1}_{I}(R[X])$. Thus using Theorem 3.1 above, we can conclude.
\end{proof}


\section{examples: when the property `$p\notin \Ass_{\ZZ} H^i_J(S)$' holds}

In this section we provide examples for the situations when `$\textit{\textbf{P}}^{i,p}_{J,S}\equiv p\notin \Ass_{\ZZ} H^i_J(S)$' is satisfied. Here we consider $S$ either as $R=\ZZ[X_1,\ldots,X_n]$ or as the unramified regular local ring $R_{\m}$ obtained from the localization of $\ZZ[X_1,\ldots,X_n]$ by the prime ideal $(p,X_1,\ldots,X_n)=\m$ i.e $\ZZ[X_1,\ldots,X_n]_{(p,X_1,\ldots,X_n)}$. The  example of the last ring is mentioned in Example 3.1. 

\begin{example}
For $R$ or $R_{\m}$ as mentioned above, $p$ is always a non-zero divisor for $H^0_I(R)$ and $H^0_I(R_{\m})$.  
\end{example}

\begin{example}
Consider the ring $S$ as $R$ or as $R_{\m}$ and the sequence of elements $f_1,\ldots,f_k$ in $S$ such that $f_1,\ldots,f_k,p$ forms a regular sequence. Set $J=(f_1,\ldots,f_k)S$. Then we have $k$-th Koszul cohomology $H^k(\textbf{f};S)=S/(f_1,\ldots,f_k)S$ and $p\notin \Ass_{\ZZ} H^k(\textbf{f};S)$. By Theorem 3.1 or Theorem 4.1 of \cite{BBLSZ} we find $p\notin \Ass_{\ZZ} H^i_J(S)$.
\end{example}

The following situation includes a large class of examples when $\textit{\textbf{P}}$ can be realized.

\begin{example}
Consider $R=\ZZ[X_1,\ldots,X_n]$ and let $I$ be an ideal of it. Fix two nonnegetive integers $i$ and $i+1$. From Theorem 5.13 of \cite{SW} (in this case, we have to consider $I$ to be generated by monomials) or more generally by Theorem 3.1 of \cite{BBLSZ}, we find that there can be finitely many $p$ torsion elements for each local cohomology module $H^i_I(R)$ and $H^{i+1}_I(R)$. Thus we can always choose a prime integer $p$ such that $p\notin \Ass_{\ZZ} H^{i}_{I}(R)\cup \Ass_{\ZZ} H^{i+1}_{I}(R)$. Now, if we localize $R$ at $(p,X_1,\ldots,X_n)=\m$ we get the unramified local ring $R_{\m}$ such that $p\notin \Ass_{\ZZ} H^{i}_{IR_{\m}}(R_{\m})\cup \Ass_{\ZZ} H^{i+1}_{IR_{\m}}(R_{\m})$.
\end{example}

\section{examples: bockstein homomorphism and ramified regular ring}

In this section, we study a few examples of Bockstein homomorphism in the context of ramified regular rings. In Example 4.1 of \cite{SW}, we observe a situation when for ramified regular ring, Bockstein homomorphism is nonzero. We find that, it is consistent with the results obtained here. Moreover, we also explore a situation when Bockstein homomorphism in the context of Example 4.1 of \cite{SW}, is zero. 

From the proof of Theorem 2.2 or Theorem 2.3, it is easy to see the following situation.

\begin{thm}
Let $R$ be a polynomial ring over $\ZZ[X,Y]$ of two variables. Consider the ramified regular local ring $T=(R/(XY-p))_{(X,Y)}$. For an ideal $J$ in $T$, let $I=J\cap R$ and for an integer $i\geq 0$, assume that $p\notin \Ass_{\ZZ} H^{i}_I(R) \cup \Ass_{\ZZ} H^{i+1}_{I}(R)$. Then for the ramified regular local ring $T$, Bockstein homomorphism $\beta^{i-1}_{p,IT}: H^{i-1}_{IT}(T/pT)\rightarrow H^{i}_{IT}(T/pT)$ is zero. 
\end{thm}

\begin{proof}
Localize $\ZZ[X,Y]$ at $(p,X,Y)$ to get the unramified local ring. Now, we can apply similar technique to `$\ZZ[X,Y]_{(p,X,Y)}$' here, which is applied to the ring `$R[[X,Y]]$' there in Theorem 2.3, to conclude for the proof.
\end{proof}

\begin{rem}
In Example 4.1 of \cite{SW}, we localize the ring $\ZZ[X,Y]/(XY-p)$ at $(X,Y)$ and we call it $T$ (as in Theorem 5.1). From the discussion of Example 4.1 of \cite{SW}, it is clear that Bockstein homomorphism $\beta^{0}_{p,IT}: H^{0}_{IT}(T/pT)\rightarrow H^{1}_{IT}(T/pT)$ is also non zero for $J=(X)T$. From Theorem 5.1, we find that $p$ should be a zero divisor i.e. $p\in \Ass_{\ZZ} H^{1}_I(R) \cup \Ass_{\ZZ} H^{2}_{I}(R)$. This can be observed by taking inverse image $I$ of $J$ in $R$, i.e. $I=J\cap R=(X, (XY-p))R= (X, p))R$. So, $p\in I$ and it a zero divisor of $H^1_I(R)$. Thus the result of Theorem 5.1 is consistent with the Example 4.1 of \cite{SW}. 
\end{rem}

Contrasting to the result of Example 4.1 of \cite{SW} we also have the following observation.

\begin{prop}
Consider the ring $T=\ZZ[X,Y]/(XY-p)$. If $I= (n)\subset \ZZ$ for integer $n\in \ZZ$, then Bockstein homomorphism $\beta^{0}_{p,IT}: H^{0}_{IT}(T/pT)\rightarrow H^{1}_{IT}(T/pT)$ is a zero map.
\end{prop}

\begin{proof}
Choose $IT=(n)T$. If $p$ is a factor in $n$ then $H^0_{IT}(T/pT)= T/pT$. This implies that $H^i_{IT}(T/pT)= 0$ for $i>0$. So Bockstein homomorphism is zero. On the other hand, if $p$ is not a factor of $n$, then we can show that $n$ is a nonzero divisor of $T/pT$. For that choose $r, r_1\in T$ and assume that $rn=pr_1$ and we have to show that $r\in pT$. Take $z, z'\in \ZZ[X,Y]$ such that $\overline{z}=r$ and $\overline{z'}=r_1$. Now $rn=pr_1$ if and only if $nz=pu_1 + XYu'_1$ for some $u, u'_1\in Z[X,Y]$. Let $z=b_{00}+ b_{10}X+ b_{01}Y + \sum_{i,j\geq 1}X^iY^j$. Since $nz=pu_1 + XYu'_1$, we get $nb_{00}, nb_{10}, nb_{01}\in (p)\ZZ$. This implies $z=pw_1 + XYw'_1$ for some $w, w'_1\in Z[X,Y]$. Since $pw_1 + XYw'_1= pv_1 + (XY-p)v'_1$ for some $w, w'_1\in Z[X,Y]$, we conclude. Thus $H^0_{IT}=0$ and Bockstein homomorphism is again a zero map.
\end{proof}

\begin{example}
In Example 4.1 of \cite{SW}, we observe that for ramified ring $T=\ZZ[X,Y]/(XY-p)$ and for ideal $(X)T\subset T$, Bockstein homomorphism $\beta^{0}_{p,(X)T}: H^{0}_{(X)T}(T/pT)\rightarrow H^{1}_{(X)T}(T/pT)$ is nonzero. In fact, if we choose any ideal $J\subset T$ such that $(X)T\subset J$, then we have the following commutative diagram 

\[
  \xymatrix
{
  & H^{0}_{(X)T}(T/pT) 
    \ar@{->}[r]
 & H^{1}_{(X)T}(T/pT) 
 \\
  & H^{0}_{J}(T/pT) 
	\ar@{->}[u]
\ar@{->}[r]
 & H^{1}_{J}(T/pT)
  \ar@{->}[u]      
 }
\]

where each horizontal map is Bockstein homomorphism. Since upper horizontal map is nonzero, we get $\beta^{0}_{p,J}: H^{0}_{J}(T/pT)\rightarrow H^{1}_{J}(T/pT)$ is also non zero.
\end{example}

We conclude the section with an example of ramified local ring when Bockstein Homomorphism is zero.

\begin{example}
Here, we mention a situation, where for an ideal of a ramified regular local ring (may not be an expanded ideal) Bockstein homomorphism is a zero map: From Example 4.3 we observe that we can always get an unramified regular local ring $R$, such that for a pair of integers $i$ and $i+1$ and for an ideal $I\subset R$, we can have $p\notin \Ass_{\ZZ} H^{i}_{I}(R)\cup \Ass_{\ZZ} H^{i+1}_{I}(R)$. Consider the ramified ring $S$ which is obtained from $R$ via Eisenstein extension. For expanded ideal $IS$, take $y\in \rad (IS +pS)$. Set $J= IS+ yS$. Thus $\rad (J+pS)= \rad (\rad (IS +pS) +\rad (yS))= \rad (IS +pS)$. Thus, by Theorem 3.1 we can conclude that Bockstein homomorphism $\beta^{i-1}_p: H^{i-1}_J(S/pS)\rightarrow H^{i}_J(S/pS)$ is zero.   
\end{example}

{\textbf{Acknowledgement:}}\newline

I would like to thank Tony J. Puthenpurakal for his invaluable comments and suggestions.


\end{document}